\documentclass[12pt, reqno]{amsart}
\usepackage{amsmath}
\usepackage{amssymb}
\usepackage{epsfig}
\usepackage{graphicx}
\usepackage{color}
\usepackage{fullpage}
\usepackage{comment}
\definecolor{shadecolor}{gray}{0.875}
\usepackage{amscd}
\usepackage{stmaryrd}

\usepackage{ulem}
\usepackage{amsthm}
\usepackage[backref=]{hyperref} 

\usepackage{csquotes}
\usepackage{enumerate}
\RequirePackage{tikz}
\usepackage{tikz-cd}

\let\oldtocsection=\tocsection

\let\oldtocsubsection=\tocsubsection

\let\oldtocsubsubsection=\tocsubsubsection

\renewcommand{\tocsection}[2]{\hspace{0em}\oldtocsection{#1}{#2}}
\renewcommand{\tocsubsection}[2]{\hspace{1em}\oldtocsubsection{#1}{#2}}
\renewcommand{\tocsubsubsection}[2]{\hspace{2em}\oldtocsubsubsection{#1}{#2}}

\usepackage{etoolbox}   
\makeatletter
\patchcmd{\@maketitle}{\artauthors}{{\artauthors}}{}{}
\makeatother

\makeatletter

\newcommand{\Rmnum}[1]{\expandafter\@slowromancap\romannumeral #1@}
\makeatother
\numberwithin{equation}{section}

\input xy
\xyoption{all}

\calclayout
\allowdisplaybreaks[3]

\theoremstyle{plain}
\newtheorem{prop}{Proposition}[section]

\newtheorem{theo}[prop]{Theorem}
\newtheorem{coro}[prop]{Corollary}

\newtheorem{lemm}[prop]{Lemma}

\theoremstyle{definition}
\newtheorem{defi}[prop]{Definition}

\newtheorem{ques}[prop]{Question}
\newtheorem{conj}[prop]{Conjecture}

\newtheorem{rema}[prop]{Remark}

\newtheorem{exam}[prop]{Example}

\def\cC{{\mathcal C}}

\def\cL{{\mathcal L}}

\def\rk{{\mathrm{rk}}}

\def\bA{{\mathbb A}}

\def\bM{{\mathbb M}}

\def\bP{{\mathbb P}}
\def\bQ{{\mathbb Q}}
\def\bR{{\mathbb R}}
\def\bZ{{\mathbb Z}}

\def\Eff{\overline{\mathrm{Eff}}}
\def\Pic{\mathrm{Pic}}

\def\Pic{\mathrm{Pic}}

\def\Sing{\mathrm{Sing}}

\def\Gal{\mathrm{Gal}}

\def\Proj{\mathrm{Proj}}
\def\ch{\mathrm{char}}

\makeatother
\makeatletter

\author{Runxuan Gao}
\address{Graduate School of Mathematics, Nagoya University, Furocho Chikusa-ku, Nagoya, 464-8602, Japan}
\email{m20015x@math.nagoya-u.ac.jp}

\title[Examples]{A Zariski dense exceptional set in Manin's Conjecture: dimension $2$}

\begin{document}

\date{\today}

\maketitle

\begin{abstract}
	Recently, Lehmann, Sengupta, and Tanimoto proposed a conjectural construction of the exceptional set in Manin's Conjecture, which we call the geometric exceptional set. We construct a del Pezzo surface of degree $1$ whose geometric exceptional set is Zariski dense.
	In particular, this provides the first counterexample to the original version of Manin's Conjecture in dimension $2$ in characteristic $0$.
	Assuming the finiteness of Tate-Shafarevich groups of elliptic curves over $\bQ$ with $j$-invariant $0$, we show that there are infinitely many such counterexamples.
\end{abstract}

\tableofcontents
\newpage
\section{Introduction}
\
Manin's Conjecture predicts that the asymptotic behavior of rational points on Fano varieties is controlled by their geometric invariants.
Let $X$ be a projective variety over a number field $F$ and let $\cL=(L,\Vert\cdot\Vert)$ be an adelically metrized ample line bundle on $X$.
Then one can associate to $\cL$ a unique height function (see, e.g., \cite{Peyre03})
$$H_\cL:X(F)\rightarrow \bR_{\geq0}.$$
The ampleness of $L$ implies that the set
$$\{ P\in X(F)\mid H_\cL(P)\leq T \}$$
is finite for any $T\in\bR$.
The counting function is defined by
\begin{equation*}
	N(Q,\cL,T):=\#\left\{P\in Q\mid H_\cL(P)\leq T\right\}
\end{equation*}
for any subset $Q\subset X(F)$.

Before stating the conjecture, we need to recall the notion of thin sets in the sense of Serre.
\begin{defi}
	Let $f:Y\rightarrow X$ be a morphism between varieties $Y$ and $X$ defined over a field $F$ and suppose $f$ is generically finite onto its image. Then $f$ is called
	\begin{enumerate}[\upshape
		(1)]
		\item a thin map of type \Rmnum{1} if $f$ is not dominant, and
		\item a thin map of type \Rmnum{2} if $f$ is dominant but not birational.
	\end{enumerate}
	A thin map is a thin map of type \Rmnum{1} or \Rmnum{2}. A thin subset of $X(F)$ is a subset of a finite union $\cup_j f(Y_j(F))$ where $f:Y_j\rightarrow X$ are thin maps of type \Rmnum{1} or \Rmnum{2}.
\end{defi}

%For any functions $f,g:\bR\rightarrow\bR$, we use the notation $f\sim g$ to indicate
%$$\lim_{T\rightarrow\infty} \frac{f(T)}{g(T)}=1.$$

A contemporary version of Manin's Conjecture is as follows.
\begin{conj}\label{conj} (Manin's Conjecture)
	Let $X$ be a smooth projective Fano variety over a number field $F$ and let $\cL=(L,\|\cdot\|)$ be an adelically metrized ample line bundle on $X$. Suppose that the set $X(F)$ is not thin. Then there exists a thin subset $Z\subset X(F)$ such that
	\begin{equation}\label{eq1.1}
		N(X(F)\backslash Z,\cL,T)\sim c(F,Z,\cL)T^{a(X,L)}\log(T)^{b(F,X,L)-1}
	\end{equation}
	as $T\rightarrow\infty$, where $a(X,L)$ and $b(F,X,L)$ are geometric invariants which we will define in Section \ref{sec3.1}, and $c(F,Z,\cL)$ is Peyre-Batyrev-Tschinkel’s constant introduced in \cite{Peyre,BT}.
\end{conj}
%The subset $Z\subset X(F)$ is known as the the exceptional set.
The exceptional set $Z$ was assumed to be a Zariski non-dense subset of $X$ in the original version of the conjecture formulated in \cite{FMT89,BM}.
However, a counterexample to this version was discovered in \cite{batyrev1996rational}.
As such, it is suggested in \cite{Peyre03} that we assume $Z$ to be a thin set,
and such a conjecture was first formulated in \cite{RUD19}.
The thin set version of Manin's Conjecture has been established in several cases, such as \cite{RUD19,BHB18,blomer2022Manin,bonolis2022density},
and no counterexample has been found yet.

Recently, Lehmann, Sengupta, and Tanimoto proposed a conjectural construction of the exceptional set $Z$ in \cite{LST18}, which we may call the geometric exceptional set to avoid ambiguity.
The techniques are formulated in a series of papers \cite{HTT15,LTT18,LTDuke,Sen21}.

The confirmed cases of Manin's Conjecture in dimension $2$ include
\begin{enumerate}[(1)]
	\item the singular cubic surface in $\bP_\bQ^3$ defined by $x_1x_2^2 +x_2x_0^2 +x_3^3 = 0$ \cite{dBBD07};
	\item toric varieties over number fields \cite{BT-0}, such as non-singular del Pezzo surfaces of degree $\geq6$;
	\item non-singular del Pezzo surfaces of degree $4$ in $\bP_\bQ^4$ with a special kind of conic bundle structure \cite{IBB11};
	\item a family of Ch\^atelet surfaces over $\bQ$ \cite{dBBP12}.
	%	, which can be blown down to a singular del Pezzo surface of degree $4$ 
\end{enumerate}
The exceptional sets in all confirmed cases in dimension $2$ are not Zariski dense.
Moreover, it is shown in \cite[Section 9]{LT19} that for general smooth del Pezzo surfaces over number fields with anticanonical height, the geometric exceptional set proposed in \cite{LST18} is not Zariski dense.

When $X$ is a del Pezzo surface of degree $1$,
the above result is based on the assumption that the curve parametrizing all rational curves in $\lvert-2K_X\rvert$ is irreducible and of genus $\geq2$,
which is only known to be satisfied when $\overline{X}$ is general in moduli by \cite{Lubbes14}, where $\overline{X}$ refers to the base change of $X$ to an algebraic closure (see \cite[Theorem II.1.2]{kollar2013rational} for the dimension of the parametrizing variety).
So it is natural to ask the following question.

\begin{ques}{\rm (\cite[Question 9.7]{LT19})}\label{ques1}
	Let $X$ be a del Pezzo surface of degree $1$ and Picard rank $1$. 
	Let $M$ denote the curve parametrizing all rational curves in $\lvert-2K_X \rvert$. Does every component of $M$ have genus greater than $1$?
\end{ques}
The original idea in \cite{LT19} is to confirm that the conjectural exceptional sets for all surfaces are Zariski non-dense by answering this question affirmatively.
But in this paper, we give a negative answer to this question by constructing a counterexample, which is summarized in the following theorem.
\begin{theo}\label{the1}
	Let $m\in\bZ^\ast$ be a nonzero integer and $S$ be the surface defined by
	\begin{equation*}
		w^2=z^3+m^2x^6+m^2y^6
	\end{equation*}
	in the weighted projective space $\bP(1,1,2,3)$ over a number field $k$. Then 
	\begin{enumerate}[\upshape(1)]
		\item $S$ is a del Pezzo surface of degree $1$.
		%		It has Picard rank $1$ when $m=7$ and $k=\bQ(\sqrt{-3})$.
		\item There exists an elliptic curve $E$ parametrizing rational curves in $\lvert-2K_S\rvert$.
		Namely, we have a diagram of varieties with canonical morphisms
		% https://q.uiver.app/?q=WzAsNCxbMSwwLCJcdFNcXHRpbWVzIEVcXHN1cHNldFxcY0MiXSxbMiwwLCJTIl0sWzEsMSwiRSJdLFswLDBdLFswLDEsIlxcbXUiXSxbMCwyLCJcXHBpIiwyLHsib2Zmc2V0IjotNX1dXQ==
		\[\begin{tikzcd}
			{} & {	S\times E\supset\cC} & S \\
			& \hspace{13mm}E
			\arrow["\mu", from=1-2, to=1-3]
			\arrow["\pi"', shift left=8, from=1-2, to=2-2]
		\end{tikzcd}\]
		such that $E$ is a smooth curve of genus $1$, $\mu$ is dominant and general fibers of $\pi$ are rational curves in $\lvert-2K_S\rvert$.
		%		\item $E$ has positive Mordell–Weil rank when $m=7$.
		\item When $k$ is a field which contains $\bQ(\sqrt{-3})$, there exists a section $s:E\rightarrow\cC$ of $\pi$, 
		which intersects all but finitely many fibers in smooth points.
	\end{enumerate}
\end{theo}

\begin{rema}
	It is worth mentioning that in $(2)$ of Theorem $\ref{the1}$, an elliptic curve has the lowest possible genus for an irreducible curve parametrizing rational curves in $\lvert-2K_S\rvert$, by \cite[Proposition 24]{Lubbes14} combined with Proposition \ref{prop1} (2).
\end{rema}

Theorem $\ref{the1}$ in particular gives a negative answer to Question \ref{ques1}.
Moreover, we can prove that the geometric exceptional set of $S$ is dense under two additional assumptions.

\begin{theo}\label{thm2}
	In the notation of Theorem \ref{the1}, suppose $k$ is a field which contains $\bQ(\sqrt{-3})$ and
	\begin{enumerate}[\upshape(1)]
		\item $E$ has positive Mordell-Weil rank, and
		\item $S$ has Picard rank $1$.
	\end{enumerate}
	Then the geometric exceptional set of $S$ is Zariski dense.
\end{theo}
%\begin{coro}\label{cor1.5}
%	When $m=7$ and $k=\bQ(\sqrt{-3})$,
%	the geometric exceptional set for $S$ given in Theorem $\ref{the1}$ is Zariski dense in $S$.
%\end{coro}
There exists an integer $m$ such that the surface $S$ satisfies these assumptions.
Moreover, if we assume a conjecture on elliptic curves, then there exist infinitely many such integers $m$.

\begin{prop}\label{prop1.6}
	The assumptions in Theorem \ref{thm2} are satisfied when $m=7$ and $k=\bQ(\sqrt{-3})$. Moreover, if we assume the Tate-Shafarevich groups of elliptic curves over $\bQ$ with $j$-invariant $0$ are finite, 
	%	then there exist infinitely many $m$ such that the corresponding surfaces $S$ are pairwise non-isomorphic and satisfy the assumptions.
	then when
	\begin{equation*}
		m\in\left\{ p: \textrm{\rm prime number}\mid p\equiv 4,7 \textrm{ \rm or }8 \,\mod\, 9 \right\},
	\end{equation*}
	the corresponding surface $S$ satisfies the assumptions.
	By Dirichlet’s theorem on arithmetic progressions, there are infinitely many such $m$.
\end{prop}

To obtain a concrete counterexample to the original version of Manin's Conjecture, we appeal to \cite[Corollary 1.3]{BL16} which shows that when a Fano variety admits a Zariski dense set of saturated subvarieties then there is a choice of adelic metric for which the exceptional set must be Zariski dense.
Thus we obtain the following corollary.
\begin{coro}\label{cor1.6}
	Suppose the surface $S$ defined in Theorem $\ref{the1}$ satisfies the assumptions in Theorem \ref{thm2}.
	Then for any non-empty open subset $U\subset S$ and any real number $A\geq 0$, there exists an adelic metric on $L$ such that
	\begin{equation*}
		\liminf_{T\rightarrow\infty}\frac{N(U,\cL,T)}{T}>(1+A)c(F,Z,\cL).
	\end{equation*}
\end{coro}
In other words, the original version of Manin's Conjecture does not hold for the surface $S$ satisfying the assumptions in Theorem \ref{thm2}. By Proposition \ref{prop1.6}, such a surface exists.
If we assume that the Tate-Shafarevich groups of elliptic curves over $\bQ$ with $j$-invariant $0$ are finite, then we have infinitely many such surfaces.

As a consequence, we obtain what appears to be the first example of a Zariski dense geometric exceptional set in dimension $2$ over fields of characteristic $0$.
It is worth mentioning that \cite{beheshti2021rational} provides the first example of a Zariski dense geometric exceptional set in dimension $2$ in positive characteristic.

The paper is organized as follows.
In Section 2 we review background knowledge and calculate the geometric exceptional sets for del Pezzo surfaces of degree $1$ and Picard rank $1$.
In Section 3 we prove Theorem \ref{the1}. In Section 4 we prove Theorem \ref{thm2}, Proposition \ref{prop1.6} and Corollary \ref{cor1.6}.
In Section 5 we try to compute the moduli space of rational curves in $\lvert-2K_X \rvert$ systematically. Although it seems difficult to determine all the irreducible components, we find three other elliptic components among them.
\bigskip

\noindent
{\bf Acknowledgments.}
The author would like to thank Sho Tanimoto for suggesting the problem and for many stimulating conversations.
The author also wishes to thank Brian Lehmann and Anthony V\'{a}rilly-Alvarado for useful comments and Julian Lyczak for pointing out an error in Proposition $\ref{prop1}$ of a previous version of this paper.
This paper is based on the author’s master's thesis at Nagoya University.
The author was partially supported by JST FOREST program Grant number JPMJFR212Z and JSPS Bilateral Joint Research Projects Grant number JPJSBP120219935.

\section{Exceptional sets in Manin's Conjecture}

\subsection{Geometric invariants in Manin’s Conjecture}
\label{sec3.1}
There are three invariants appearing in Conjecture $\ref{conj}$ and we will discuss $a(X,L)$ and $b(F,X,L)$ in this subsection.  

We first recall the definition of a pseudo-effective cone. For any projective variety $X$ there is an intersection pairing between $N^1(X):=N^1(X)_\bZ\otimes \bR$ and $N_1(X):=N_{1}(X)_\bZ\otimes \bR$, where $N^1(X)_\bZ$ and $N_{1}(X)_\bZ$ are the quotient of the group of Cartier divisors and of $1$-cycles by numerical equivalence, respectively.
We attach to $N^1(X)$ and $N_1(X)$ the Euclidean topology.
Then the pseudo-effective cone $\Eff^1(X)$ is defined to be the closure of the convex cone generated by classes of effective divisors.
It turns out that the cone of big divisors equals the interior of $\Eff^1(X)$.

\begin{defi}
	Let $X$ be a smooth projective variety and let $L$ be a big and nef divisor. The $a$-invariant or Fujita invariant is defined to be
	$$a(X,L):=\min\{ t\in\bR\mid K_X+tL\in \Eff^1(X)\}.$$
	%	When $X$ is a non-smooth variety defined over a field of characteristic zero,
	In characteristic zero, the $a$-invariant of a non-smooth variety is defined by a resolution of singularities, see Remark $\ref{rem}$.
\end{defi}
\begin{exam}
	If $X$ is a Fano variety and $L=-K_X$, then we have $a(X,L)=1$.
\end{exam}
\begin{exam}
	If $C$ is an (irreducible but not necessarily smooth) rational curve and $L$ a big and nef Cartier divisor on it, then we have $a(C,L)= 2/(C\cdot L)$, since the normalization of $C$ is birational to $\bP^1$ and a divisor $D$ is contained in $\Eff^1(\bP^1)$ if and only if $\deg(D)\geq0$.
\end{exam}
Let $K$ be a closed convex cone in a finite-dimensional real vector space $V$.
Recall that a subcone $C\subseteq K$ is called a face if it has the property that if $u+v\in C$ for some vectors $u,v\in K$, then $u\in C$ and $v\in C$. A face $C$ of a cone $K$ is said to be supported if there exists a linear functional $l: V\rightarrow \bR$ such that $l(v)\geq 0$ for every $v\in K$ and $C=\left\{ v\in K\mid l(v)=0 \right\}$.
\begin{defi}
	Let $X$ be a smooth geometrically integral projective variety over a field $F$ and let $L$ be a big and nef divisor on it. The $b$-invariant $b(F,X,L)$ is defined to be the codimension of the minimal supported face of $\Eff^1(X)$ containing $K_X+a(X,L)L$.
	
	In characteristic zero, the $b$-invariant of a non-smooth variety is defined by a resolution of singularities, see Remark $\ref{rem}$.
\end{defi}

\begin{exam}
	Let $X$ be a smooth Fano variety over a field $F$ and take $L=-K_X$, then $b(F,X,L)$ is equal to the Picard rank of $X$.
\end{exam}
For more examples and properties of $a$- and $b$-invariants, see \cite[Sections 3 and 4]{LT19}.
\begin{rema}\label{rem}
	When $X$ is a singular projective variety defined over a field $F$ of characteristic zero, we can not use the same definition to define $a(X,L)$ and $b(F,X,L)$ since the canonical divisor $K_X$ may not exist.
	Instead, we take a resolution of singularities $\phi:X^\prime\rightarrow X$ and set $a(X,L):=a(X^\prime,\phi^\ast L)$ and $b(F,X,L):=b(F,X^\prime,\phi^\ast L)$.
	These are well-defined by \cite[Proposition 2.7, Proposition 2.10]{HTT15}.
\end{rema}

\subsection{Adjoint rigid pair}\label{secadj}
\

For a smooth projective variety $X$ and a big and nef divisor $L$, we say the pair $(X,L)$ is adjoint rigid if for any sufficiently divisible positive integer $m$, we have $\dim H^0(X,m(K_X+a(X,L)L))=1$ \cite[Definition 3.10]{LT19}.

Let $X$ be a smooth geometrically uniruled projective variety and let $L$ be a big and nef divisor. Let $m$ be a sufficiently divisible positive integer such that $m(K_X+a(X,L)L)$ is a Cartier divisor.
It has been shown in \cite{BCHM} that the section ring
$$R(X,K_X+a(X,L)L):=\bigoplus_{d\geq0}H^0(X,md(K_X+a(X,L)L))$$
is finitely generated. Thus one can define a projective variety $Z:=\Proj\, R(X,K_X+a(X,L)L)$ and a rational map $\pi:X\dashrightarrow Z$ associated to the linear system $\lvert m(K_X+a(X,L)L) \rvert$, which is known as the canonical map for the pair $(X,a(X,L)L)$.
It follows that $(X,L)$ is adjoint rigid if and only if $\pi$ maps $X$ to a point.

\subsection{Geometric exceptional sets of Fano varieties}\label{sec3.2-}
%\
%
Motivated by geometry, a conjectural construction of the exceptional set in Manin's Conjecture has been proposed in \cite[Section 5]{LST18}, which we call the geometric exceptional set. Although the construction works for a more general case, we will concentrate on the case of Fano varieties.
\begin{defi}\label{def_ges}
	Let $X$ be a smooth Fano variety over a number field and take $L=-K_X$.
	Then the geometric exceptional set $Z$ is the union of $f(Y(F))$ where $f:Y\rightarrow X$ varies over all thin maps in the following cases.
	\begin{enumerate}[(1)]
		\item When $(a(Y,f^\ast L),b(F,Y,f^\ast L))>(a(X,L),b(F,X,L))$ in lexicographical order.
		\item When $(a(Y,f^\ast L),b(F,Y,f^\ast L))=(a(X,L),b(F,X,L))$ and either
		\begin{enumerate}[(a)]
			\item $\dim Y<\dim X$, or
			\item $\dim Y=\dim X$ and $(Y,f^\ast L)$ is not adjoint rigid, or
			\item $\dim Y=\dim X$, the pair$(Y,f^\ast L)$ is adjoint rigid and $f$ is face contracting.
		\end{enumerate}
		%		 we include the points contributed by $Y$, except for those $Y$ such that the following three conditions are satisfied simultaneously: $\dim(Y)=\dim(X)$, the pair $(Y,f^\ast L)$ is adjoint rigid, and $f$ is not face contracting (see \cite[Definition 3.5]{LT19a} for the definition of face contracting morphisms).
	\end{enumerate}
\end{defi}
See \cite[Definition 3.5]{LT19a} for the definition of face contracting morphisms.
%Let $X$ be a smooth projective geometrically integral Fano variety over a number field and take $L=-K_X$.
%Then the geometric exceptional set $Z$ constructed in \cite{LST18} is the union of $f(Y(F))$ where $f:Y\rightarrow X$ varies over all thin maps from a smooth variety $Y$ satisfying certain geometric conditions.
%When $(a(X,L),b(F,X,L))<(a(Y,f^\ast L),b(F,Y,f^\ast L))$ in lexicographical order, we include the points contributed by $Y$ in $Z$.
%When $(a(X,L),b(F,X,L))=(a(Y,f^\ast L),b(F,Y,f^\ast L))$, we include the points contributed by $Y$, except for those $Y$ such that $\dim(Y)=\dim(X)$, $\kappa(K_Y+a(Y,f^\ast L)f^\ast L)=0$, and $f$ is not face contracting.

%such that $(a(X,L),b(F,X,L))\leq(a(Y,f^\ast L),b(F,Y,f^\ast L))$ in lexicographical order, but without the case when $\dim(Y)=\dim(X)$, $\kappa(K_Y+a(Y,f^\ast L)f^\ast L)=0$, and $f$ is not face contracting,
%where face contraction is a technical property defined in \cite[Definition 3.5]{LT19a}.

\subsection{Geometric exceptional sets for del Pezzo surfaces of degree $1$ and Picard rank $1$}\label{sec3.2}
\

%Let $X$ be a del Pezzo surface of degree $1$ and Picard rank $1$ defined over a number field $F$ and set $L=-K_X$. Then we have
%$$a(X,L)=1,\quad b(F,X,L)=1.$$
In this section we compute the geometric exceptional set for del Pezzo surface of degree $1$ and Picard rank $1$. The result is as follows.

\begin{prop}\label{propgesofdp1}
	let $X$ be a del Pezzo surface of degree $1$ and Picard rank $1$ defined over a number field, then the geometric exceptional set for $X$ is the union of rational points on rational curves in $\lvert-K_X \rvert$ and $\lvert-2K_X \rvert$.
\end{prop}

%The exceptional set is the union of images of some thin maps $f:Y\rightarrow X$ and there are three cases for $Y$ which contribute to the exceptional set, see \cite[Conjecture 5.2]{LT19}:
\begin{proof}
	In the notation of Definition \ref{def_ges}, the dimension of $Y$ is $1$ or $2$.
	
	When $\dim Y=2$, we have $a(Y,f^\ast L)\leq a(X,L)$ by \cite[Lemma 3.12]{LT19}.
	So only $Y$ with $a(Y,f^\ast L)= a(X,L)=1$ may have contributions to the geometric exceptional set.
	Suppose $(Y, f^\ast L)$ is not adjoint rigid,
	then the canonical map for the pair $(Y, a(Y,f^\ast L)f^\ast L)$ maps $Y$ to a curve,
	since the divisor $a(Y,f^\ast L)f^\ast L+K_Y$ is on the boundary of $\Eff^1(X)$ and so is not big.
	By \cite[Lemma 3.38 and Corollary 3.53]{KM98},
	%$Y$ is birational fibered by curves with the same $a$-value. 
	a general fiber of the canonical map has the same $a$-value as $Y$ and is adjoint rigid with respect to $f^\ast L$.
	So we reduce to the case where $Y$ is a curve and $(Y,f^\ast L)$ is adjoint rigid.
	Thus in the case when $\dim Y=2$, we only need to consider $Y$ such that $(Y, f^\ast L)$ is adjoint rigid and $a(Y,f^\ast L)= a(X,L)=1$,
	but there is no such variety over $X$ by \cite[Theorem 6.2]{LTDuke}.

	Thus we reduce to the case $\dim Y=1$.
	Let $f:C\rightarrow X$ be a curve on $X$.
	Suppose $C$ is not a rational curve,
	%Let $C$ be a curve on $X$ and $f:C\rightarrow X$ be the inclusion map.
	%Suppose $C$ is not rational,
	then we have
	$$\dim\lvert K_C\rvert=g(C)-1\geq 0,$$
	So there exists an effective divisor linearly equivalent to $K_C$.
	Thus for any non-rational curve $C$ on $X$, we have $a(C,f^\ast L)\leq0<a(X,L)$ and thus $C$ make no contribution to the geometric exceptional set.
	%	, which implies that $C$ does not contribute any rational points to the geometric exceptional set for $X$.
	
	So we only need to consider rational curves on $X$. Suppose $C$ is a rational curve, then we have $a(C,f^\ast L)=2/(C\cdot f^\ast L)$.
	So the condition $a(X,L)\leq a(C,f^\ast L)$ is equivalent to $C\cdot K_X\geq -2$.
	Since $-K_X$ is ample, by the Nakai–Moishezon Criterion we only have the following cases.
	\begin{enumerate}[\upshape(1)]
		\item If $C$ is a rational curve with $C\cdot K_X=-1$, then $C\in\lvert-K_X \rvert$ and $a(C,f^\ast L)=2$.
		\item If $C$ is a rational curve with $C\cdot K_X=-2$, then $C\in\lvert-2K_X \rvert$ and $a(C,f^\ast L)=1$.
	\end{enumerate}
	%	Note that in case (2), the $a$-invariants of $C$ and $X$ are equal, so we further need to compare $b$-invariants of them.
	%%	But we automatically have $b(F,C^\prime,f^\ast L)=b(F,X,L)=1$ for any rational curves $C^\prime$. 
	%	But we always have $b(F,C,f^\ast L)=b(F,X,L)=1$.
	In both cases, we have $(a(X,L),b(F,X,L))\leq(a(C,f^\ast L),b(F,C,f^\ast L))$ and $\dim C<\dim X$.
	Thus the geometric exceptional set precisely consists of rational points on curves of the above two cases.
\end{proof}

%Note that we can not remove the condition that $X$ is of Picard rank $1$, otherwise $b(F,X,L)>1$ and rational curves on $X$ are not included in the exceptional set.
Note that if the Picard rank of $X$ was $\geq2$, then we would have $b(F,X,L)> b(F,C,f^\ast L)=1$ and so rational curves in $\lvert-2K_X \rvert$ would not be contained in the geometric exceptional set of $X$. So it is necessary to check that the Picard rank of $X$ is $1$ in Theorem $\ref{the1}$.

Next, we want to determine whether the geometric exceptional set is dense or not.

\begin{prop}
	Let $X$ be a del Pezzo surface of degree $1$ and Picard rank $1$.
	Then the family of rational curves in $\lvert-K_X \rvert$ will never contain a dense set of rational points, and neither will the family in $\lvert-2K_X \rvert$ for general $X$.
\end{prop}
\begin{proof}
	Let  $C\in\lvert-K_X \rvert$ be a smooth curve.
	The adjunction formula gives
	$$2g(C)-2=K\cdot C+C^2,$$
	so the genus of $C$ is $1$. 
	Thus by Bertini's Theorem, general members in $\lvert-K_X \rvert$ are elliptic curves. So there exist only finitely many rational curves in $\lvert-K_X \rvert$,
	%	on which rational points cannot be dense in $X$.
	and the set of all rational points on these curves is not dense in $X$.
	
	When $C$ is a rational curve in $\lvert-2K_X \rvert$, if we further suppose that $\overline{X}$ is general in moduli, then it is shown in \cite{LST18} that such curves are parametrized by an irreducible curve of geometric genus $\geq 2$, on which set of rational points also can not be Zariski dense.
\end{proof}
The case where $\overline{X}$ is not general in moduli remained open. It is worth mentioning that \cite[Proposition 24]{Lubbes14} combined with Proposition \ref{prop1} (2) shows that there is no rational curve parametrizing an irreducible family of rational curves in $C\in\lvert-2K_X \rvert$, so the problem is to determine whether there exists an elliptic one.

%Summarizing, let $X$ be a del Pezzo surface of degree $1$ and Picard rank $1$ defined over a number field, then the exceptional set for $X$ is the union of rational points on rational curves in $\lvert-K_X \rvert$ and $\lvert-2K_X \rvert$.

%\section{Proof of Theorem $1.3$}
\section{Construction of the examples}
\subsection{Bitangent correspondence of rational curves in $\lvert-2K_S\rvert$}\label{sec4.1}
\ 

Let $S$ be a del Pezzo surface of degree $1$ over an algebraic closed field $k$ of characteristic zero. Then the linear system $\lvert-2K_S\rvert$ is base-point-free
and has dimension $3$.
The associated morphism $\varphi:S\rightarrow \bP^3$ is a double cover of a quadric cone $Q\subset\bP^3$ branched along the vertex $V$ of $Q$ and a smooth curve $B$. 
%Our goal in this section is to show that rational curves in $\lvert-2K_S\rvert$ correspond to planes in $\bP^3$ which are bitangent to $B$.
%Here by a bitangent plane of $B$ we mean there exist two points in $B\cap H$ with intersection multiplicity $\geq2$ or one point with intersection multiplicity $\geq4$.

\begin{prop}\label{prop1}
	With notation as above, we have
	\begin{enumerate}[\upshape(1)]
		\item $B$ is a smooth curve on $Q$ cut out by a cubic surface.
		\item $B$ is of degree $6$ and genus $4$.
		\item General rational curves in $\lvert-2K_S\rvert$ correspond to hyperplane sections $H\vert_B=m_1p_1+\cdots+m_lp_l$ with $V\notin H$ where $$(m_1,\cdots,m_l)=(2,2,1,1),(3,2,1),(3,3),(4,1,1)\textrm{ or }(5,1)$$ up to rearranging the indices.
		In particular, $\varphi^\ast H$ is a rational curve in $\lvert-2K_S\rvert$ when $H$ is bitangent to $B$.
		
		%		General rational curves in $\lvert-2K_S\rvert$ correspond to planes $H$ in $\bP^3$ which are bitangent to $B$.
	\end{enumerate}
\end{prop}
\begin{proof}
	%		\item This is \cite[Theorem 8.3.2]{Dol12}, where $B$ is smooth since $S$ is smooth.
	%		\item Since the branch curve $B$ is smooth, the adjunction formula gives
	%		$$2g-2=K_Q\cdot B+B^2,$$
	%		so the genus of $B$ is $g=4$.
	%		Let $H$ denote a plane in $\bP^3$. Then the degree of $B$ is
	%		$$B\cdot H=3H\cdot Q\cdot H=6.$$
	(1) and (2) are proved in \cite[Proposition 3.1]{dP1char2}.
	To prove (3), let $C\in\lvert-2K_S\rvert$ be a general irreducible curve and $\varphi^\prime:C\rightarrow H\vert_Q$ be the restriction of $\varphi$ to $C$.
	Let $\tilde{C}\rightarrow C$ be the normalization and $R^\prime$ denote the ramification divisor of $\tilde{C}\rightarrow H\vert_Q$.
	We may assume $H\vert_Q$ is smooth since $C$ is a general member.
	So we have $g(H\vert_Q)=0$.
	From the Hurwitz formula
	$$2g(\tilde{C})-2=2(g(H\vert_Q)-2)+\deg R^\prime,$$
	we obtain
	$$2g(\tilde{C})+2=\deg R^\prime.$$
	Thus $C$ is a rational curve if and only if $\deg R^\prime=2$.
	%		It follows from (2) that $D\cdot B=6$. Let $p\in D$ be a closed point and $p^\prime$ be the pull-back of $p$ by $\varphi^\prime$.
	%		One checks $\deg R^\prime=2$ if and only if $B$ and $D$ have exactly $2$ intersections of multiplicity $2$, which means $H$ is a bitangent plane of $B$.
	
	Let $\varphi(P)\in H\vert_Q\cap B$ be an intersection point with intersection number $m$. By inverse function theorem, there exists open neighborhoods $U_{\varphi(P)}$ and $U_P$ of $\varphi(P)$ and $P$ respectively such that there exists a commutative diagram
	\[\begin{tikzcd}
		U_P \arrow[rrrr, "\varphi\vert_{U_P}"] \arrow[d]               &  &  &  & U_{\varphi(P)} \arrow[d] \\
		\bA^2 \arrow[rrrr, "{(x,y)\mapsto(s,t)=(x^2,y)}"] \arrow[rrrr] &  &  &  & \bA^2      
	\end{tikzcd}\]
	where the vertical morphisms are \'{e}tale, and we may assume the curves corresponding to $B$ and $H\vert_Q$ in $\bA^2$ are $s=0$ and $s-t^m=0$.
	The ramification index of $\varphi^\prime$ at $P$ is $2$ when $m$ is odd and is $1$ when $m$ is even.
	%		The curve $\varphi^\ast (H\vert_Q)$ is locally irreducible in the previous case and is locally reducible in the latter case.
	The above local calculation shows that $\deg R^\prime$ is equal to the number of intersection points of $H$ and $B$ with odd multiplicity.
	It follows from (2) that $H\cdot B=6$.
	So the possible intersection numbers for $\deg R^\prime =2$ are $(2,2,1,1),(3,2,1),(3,3),(4,1,1)$ and $(5,1)$.
	Conversely, the curve $\varphi^\ast (H\vert_Q)$ is irreducible in all these cases, since if it is reducible, then the intersection multiplicity at each point in $H\vert_Q\cap B$ must be even.
	The assumption $V\notin H$ implies that $H\vert_Q$ is smooth.
	Thus we conclude.
	%		$V\notin H$ implies that the singular points of $\varphi^\ast (H\vert_Q)$ can only be the intersections between the two irreducible components.
	\qedhere
\end{proof}

\subsection{An elliptic family of bitangent planes}\label{sec4.2}
\

In this section we prove the first two parts of Theorem $\ref{the1}$.
When $m$ is a nonzero integer, then $S$ defined in Theorem \ref{the1} is a smooth surface.
It is well-known that such a smooth surface is a del Pezzo surface of degree $1$.
The second part is summarized as follows.
\begin{prop}[Theorem \ref{the1}(2)]\label{proppart2}
	With notation as in Theorem $\ref{the1}$, we have a diagram of varieties with canonical morphisms
	% https://q.uiver.app/?q=WzAsNCxbMSwwLCJcdFNcXHRpbWVzIEVcXHN1cHNldFxcY0MiXSxbMiwwLCJTIl0sWzEsMSwiRSJdLFswLDBdLFswLDEsIlxcbXUiXSxbMCwyLCJcXHBpIiwyLHsib2Zmc2V0IjotNX1dXQ==
	\[\begin{tikzcd}
		{} & {	S\times E\supset\cC} & S \\
		& \hspace{13mm}E
		\arrow["\mu", from=1-2, to=1-3]
		\arrow["\pi"', shift left=8, from=1-2, to=2-2]
	\end{tikzcd}\]
	where $$E: z_0^{3}+m^2x_0^{3}+m^2y_0^{3}=0 \subset \bP^2$$
	is a smooth curve of genus $1$ and $\cC$ is the subvariety of $S\times E$ cut out by the equation
	$$z_0^{2}z+m^2x_0^{2}x^2+m^2y_0^{2}y^2=0.$$
	Moreover, $\mu$ is dominant and general fibers of $\pi$ are rational curves in $\lvert-2K_S\rvert$.
\end{prop}

%For simplicity we let $m=7$.
%Let $k$ be a field of characteristic $0$ and $S$ be the surface in $\bP_k(1,1,2,3)$ given by
%$$w^2=z^3+49x^6+49y^6.$$
%Then $S$ is a del Pezzo surface of degree $1$.
\begin{proof}
	Let $Q$ be the quadric cone in $\bP^3$ defined by $Y^2=XZ$.
	Then we have the following isomorphism
	$$\bP(1,1,2)\cong Q$$
	sending $(x:y:z)\in\bP(1,1,2)$ to $(x^2:xy:y^2:z)=(X:Y:Z:W)\in\bP^3$.
	The morphism defined by the linear system $\lvert-2K_S\rvert$ can be factored as
	$$S\overset{\varphi}{\rightarrow} \bP(1,1,2)\cong Q\subset \bP^3,$$
	where $\varphi$ maps $(x:y:z:w)\in \bP(1,1,2,3)$ to $(x:y:z)\in \bP(1,1,2)$ and is branched along the vertex of $Q$ and a smooth curve
	$$B=V(z^3+m^2x^6+m^2y^6)\subset \bP(1,1,2),$$
	which can be embedded in $\bP^3$ as
	$$B=V(W^3+m^2X^3+m^2Z^3)\cap V(XZ-Y^2)\subset \bP^3.$$
	%Note that $B$ does not pass through the vertex $(0,0,1)\in \bP(1,1,2)$ of the quadric cone, so Proposition $\ref{prop1}$ applies.
	Let $\iota:\bP^3\rightarrow\bP^3$ be the involution sending $(X:Y:Z:W)$ to $(X:-Y:Z:W)$. Then $\iota$ restricted to $B$ induces an involution of $B$.
	For each closed point $P_0=(X_0:Y_0:Z_0:W_0)\in B$, the tangent planes of $V(W^3+m^2X^3+m^2Z^3)$ at $P_0=(X_0:Y_0:Z_0:W_0)$ and $\iota(P_0)=(X_0:-Y_0:Z_0:W_0)$ are the same plane
	\begin{align*}
		H_{P_0}:=V(W_0^2W+m^2X_0^2X+m^2Z_0^2Z)\subset \bP^3,
	\end{align*}
	%	and thus a bitangent plane of $B$ with tangent points $P_0$ and $\iota(P_0)$.
	which is generally bitangent to $B$ and does not contain the vertex of $Q$.
	Its intersection with $\bP(1,1,2)$ is the curve
	$$C_{(x_0^\prime:y_0^\prime:z_0^\prime)}:=V(z_0^{\prime 2}z+m^2x_0^{\prime 4}x^2+m^2y_0^{\prime 4}y^2)\subset \bP(1,1,2),$$
	where $(x_0^\prime:y_0^\prime:z_0^\prime)$ satisfies
	$$E^\prime: z_0^{\prime3}+m^2x_0^{\prime 6}+m^2y_0^{\prime 6}=0.$$
	%The rational curve given by $P_0=(x_0,y_0,z_0)\in B$ is
	%$$\varphi^\ast C_{(x_0,y_0,z_0)}=V(z_0^2z+mx_0^4x^2+my_0^4y^2)\subset S.$$
	The pullback of $C_{(x_0^\prime:y_0^\prime:z_0^\prime)}$ by $\varphi$ is the corresponding rational curve on $S$.
	%	Different points $(x_0^\prime,y_0^\prime,z_0^\prime)\in E^\prime$ may correspond to the same rational curve on $S$.
	%	But if we change the coordinates by $(x_0,y_0,z_0)=(x_0^{\prime2},y_0^{\prime2},z_0)$. 
	%	Then one checks that this family of rational curves is parametrized by the Fermat cubic
	%	$$E: z_0^{3}+mx_0^{3}+my_0^{3}=0$$
	Changing either $x_0^\prime$ to $-x_0^\prime$ or $y_0^\prime$ to $-y_0^\prime$ does not change the corresponding rational curve on $S$, so we can take the change of coordinates $(x_0^{\prime2}:y_0^{\prime2}:z_0^\prime)=(x_0:y_0:z_0)$
	and the result is the Fermat cubic $E$. Finally, $\mu$ is dominant since the family of curves are movable on $S$.
\end{proof}

%by sending $P_0^\prime=(x_0,y_0,z_0)$ to $P_0=(x_0^\prime,y_0^\prime,z_0^\prime)=(x_0^2,y_0^2,z_0)$.
%The corresponding rational curves are
%$$C^\prime_{(x_0^\prime,y_0^\prime,z_0^\prime)}:=V(z_0^{\prime2}z+49x_0^{\prime2}x^2+49y_0^{\prime2}y^2)\subset S.$$

%\subsection{Mordell-Weil rank of the elliptic curve}\label{secrank}
%\
%
%To produce a Zariski dense set of rational points on $E$, we need to show that the Mordell-Weil rank of $E$ is at least $1$, sufficiently over the field $\bQ$.
%
%By a change of variables, the curve $E$ can be written as
%$$E_m:x^3+y^3-mz^3=0,$$
%where $m=7$ in our case.
%Substituting $z=1$ and 
%\begin{equation*}
%	X=\frac{12m}{y+x},\quad Y=36m\frac{y-x}{y+x},
%\end{equation*}
%we obtain a Weierstrass form of the elliptic curve $E_m:Y^2=X^3-432m^2$. (See \cite[Introduction]{Jed05})
%Using the computer software package Magma \cite{magma}, we find that when $m=7$, the elliptic curve $E_m$ is of Mordell-Weil rank $1$.

\subsection{A section}\label{secsec}
\

Assume the ground field $k$ is a field of characteristic $0$ containing a cube root of unity.
We prove (3) of Theorem \ref{the1} in this section.
\begin{prop}[Theorem \ref{the1}(3)]\label{proppart3}
	Let $s:E\rightarrow\cC$ be the section of $\pi$ given by
	\begin{equation*}
		x=\zeta y_0,\quad
		y=x_0,
	\end{equation*}
	where $\zeta$ is a cubic root of unity. Then $\mu(s(p))$ is a smooth rational point on $S$ for a general rational point $p\in E$.
\end{prop}
\begin{proof}
	The rational curve $C_{(x_0:y_0:z_0)}\subset S\subset \bP(1,1,2,3)$ corresponding to $(x_0,y_0,z_0)\in E$ is
	\begin{equation*}
		\begin{cases}
			w^2=m^2x^6+m^2y^6+z^3\\
			m^2x_0^2x^2+m^2y_0^2y^2+z_0^2z=0.
		\end{cases}
	\end{equation*}
	Substituting the second equation into the first one gives
	$$w^2=m^2x^6+m^2y^6-\frac{m^6(x_0^2x^2+y_0^2y^2)^3}{z_0^6}.$$
	Since $(x_0,y_0,z_0)\in E$, we have
	$$z_0^3+m^2x_0^3+m^2y_0^3=0.$$
	This yields
	\begin{align*}
		\frac{(x_0^3+y_0^3)^2w^2}{m^2}=&(x_0^3+y_0^3)^2(x^6+y^6)-(x_0^2x^2+y_0^2y^2)^3\\
		=&(y_0x^2-x_0y^2)^2[(2x_0^3y_0+y_0^4)x^2+(x_0^4+2x_0y_0^3)y^2].
	\end{align*}
	Now the multi-section $(x^2,y^2)=(x_0,y_0)$ corresponds to the two singular points of the rational curve $C_{(x_0:y_0:z_0)}$. Hence it is sufficient to find another section of $\pi$.
	In other words, we need an explicit choice of two polynomials $x$ and $y$ in $x_0,y_0$ to make $F:=(2x_0^3y_0+y_0^4)x^2+(x_0^4+2x_0y_0^3)y^2$ a perfect square.
	
	For this to happen, we set $k=\bQ(\zeta)$ where $\zeta$ is a primitive cube root of unity. And we claim that $(x,y)=(\zeta y_0, x_0)$ meets the requirement. Indeed, substituting it into $F$ yields
	\begin{align*}
		F=&(2x_0^3y_0+y_0^4)\zeta^2y_0^2+(x_0^4+2x_0y_0^3)x_0^2\\
		=&x_0^6+2(\zeta^2+1)x_0^3y_0^3+\zeta^2y_0^6\\
		=&x_0^6-2\zeta x_0^3y_0^3+\zeta^2 y_0^6\\
		=&(x_0^3-\zeta y_0^3)^2,
	\end{align*}
	which is a perfect square.
	The corresponding expressions of $w$ and $z$ are
	$$w=\pm\frac{m(y_0^3-\zeta^2x_0^3)(x_0^3-\zeta y_0^3)}{(x_0^3+y_0^3)},\quad
	z=\frac{m^2\zeta x_0^2y_0^2}{z_0^2}.$$
\end{proof}

\subsection{Galois action and Picard rank}\label{secgalois}
\

Our goal in this section is to show that the surface $S$ defined in Theorem $\ref{the1}$ has Picard rank $1$ when $m=7$ and $k=\bQ(\sqrt{-3})$.

In \cite{VA09}, Várilly-Alvarado studied the del Pezzo surfaces in $\bP(1,1,2,3)$ of the form
$$w^2=z^3+Ax^6+By^6.$$
Let $X$ be such a surface over an arbitrary field $k$ with $\ch k\neq 2,3$.
He showed that the splitting field of $X$ is $K:=k(\zeta,\sqrt[3]{2},\alpha,\beta)$,
where $\zeta$ is a cube root of unity and $\alpha,\beta$ are sixth roots of $A,B$ respectively.
He also computed the intersection matrix of the $240$ exceptional curves of $X$, and gave specific equations for an orthogonal basis of $\Pic\, X_K$, which we used to compute the intersection matrices $\bM_\sigma$ and $\bM_\rho$ below.

Now set the ground field $k$ to be $\bQ(\zeta)=\bQ(\sqrt{-3})$.
Then the splitting field of $S$ is $K:=k(\sqrt[3]{2},\sqrt[3]{7})$. 
Let $\Gamma_i$ $(i=1,\cdots,9)$ be the curves defined in \cite[Section 5]{VA09}. Then the classes of $\Gamma_i$ $(i=1,\cdots,8)$ and $\Gamma_9+\Gamma_1+\Gamma_2$ form a basis of the free abelian group $\Pic\, S_K$ \cite[Proposition 5.2]{VA09}.

The Galois group $G:=\Gal(K/k)$ is generated by the two elements
\begin{align*}
	\sigma:\sqrt[3]{2}\mapsto \zeta\sqrt[3]{2},\quad\rho:\sqrt[3]{7}\mapsto\zeta\sqrt[3]{7}.
\end{align*}
%The group $G$ acts on $\Pic\, S$ by its group action on each coefficient. 
We use Magma to compute the intersection matrices $\bM_\sigma$ and $\bM_\rho$ of the above basis of $\Pic\, S_K$ after the group action by $\sigma$ and $\rho$ respectively, and the later one is
%\begin{align*}
%	\bM_\sigma=&\left(
%	\begin{array}{ccccccccc}
	%		1 & 0 & 0 & 0 & 0 & 0 & 0 & 0 & 0\\
	%		0 & 1 & 0 & 0 & 0 & 0 & 0 & 0 & 0\\
	%		0 & 0 &-1 &-1 & 0 &-1 &-1 &-1 & 2\\
	%		0 & 0 & 0 &-1 &-1 &-1 &-1 &-1 & 2\\
	%		0 & 0 &-1 & 0 &-1 &-1 &-1 &-1 & 2\\
	%		0 & 0 & 0 & 0 & 0 & 0 &-1 &-1  &1\\
	%		0 & 0 & 0 & 0 & 0 &-1 & 0 &-1  &1\\
	%		0 & 0 & 0 & 0 & 0 &-1 &-1 & 0  &1\\
	%		0 & 0 &-1 &-1 &-1 &-2 &-2 &-2 & 4\\
	%	\end{array}
%	\right),\\
%	\bM_\rho=&\left(\begin{array}{ccccccccc}
	%		-2 &-3 &-2 &-2 &-2 &-2 &-2 &-2 & 6\\
	%		0 &-2 &-1 &-1 &-1 &-1 &-1 &-1 & 3\\
	%		-1 &-2 &-2 &-2 &-1 &-2 &-2 &-2 & 5\\
	%		-1 &-2 &-1 &-1  &0 &-1 &-1 &-1 & 3\\
	%		-1 &-2 &-2 &-2 &-1 &-1 &-1 &-1 & 4\\
	%		-1 &-2 &-1 &-2 &-1 &-2 &-1 &-1 & 4\\
	%		-1 &-2 &-1 &-2 &-1 &-1 &-2 &-1 & 4\\
	%		-1 &-2 &-1 &-2 &-1 &-1 &-1 &-2 & 4\\
	%		-3 &-6 &-4 &-5 &-3& -4 &-4 &-4 &12\\
	%	\end{array}
%	\right).
%\end{align*}
\begin{align*}
	\bM_\rho=&\left(
	\begin{array}{ccccccccc}
		-2&  0& -1& -1& -1& -1& -1& -1&  3\\
		-3& -2& -2& -2& -2& -2& -2& -2&  6\\
		-2& -1& -2& -2& -1& -2& -2& -2&  5\\
		-2& -1& -1& -2& -2& -2& -2& -2&  5\\
		-2& -1& -2& -1& -2& -2& -2& -2&  5\\
		-2& -1& -1& -1& -1& -2& -2& -1&  4\\
		-2& -1& -1& -1& -1& -1& -2& -2&  4\\
		-2& -1& -1& -1& -1& -2& -1& -2&  4\\
		-6& -3& -4& -4& -4& -5& -5& -5& 13
	\end{array}
	\right).
\end{align*}
The subspace fixed by $\bM_\rho$ is the one-dimensional subspace generated by the canonical divisor
$$K_S=( 1,  1,  1,  1,  1,  1,  1,  1, -3).$$
%By computing the eigenvalues and eigenspaces of these matrices, we found that the subspace fixed by both matrices is a one-dimensional subspace generated by the vector
%$$( 1,  1,  1,  1,  1,  1,  1,  1, -3),$$
%which is nothing but the canonical divisor of $S_K$.
It is a general fact for the splitting field $K$ of $S$ that
\begin{equation}
	\Pic\, S\cong (\Pic\, S_{K})^G,
\end{equation}
and $K_S$ is always fixed by $G$.
Thus we conclude that the Picard rank of $S$ (over the number field $k=\bQ(\sqrt{-3})$) is $1$.

\subsection{The proof}
We summarize the proof of Theorem $\ref{the1}$ as follows.
\begin{proof}[Proof of Theorem \ref{the1}]
	Part (1) is discussed at the start of Section \ref{sec4.2},
	part (2) is Proposition \ref{proppart2}, and part (3) is Proposition \ref{proppart3}.
	%	The Picard rank of $S$ is computed in Section $\ref{secgalois}$. 
	%	By Proposition $\ref{prop1}$, rational curves in $\lvert-2K_S\rvert$ correspond to bitangent planes of the curve $B$, and a family of such planes is described in Section $\ref{sec4.2}$, which is parametrized by an elliptic curve $E$.
	%	The morphism $\mu$ is dominant since for any closed point $P\in\bP^3$ there is a point $P_0\in B$ such that $P\in H_{P_0}$ by the equation of $H_{P_0}$.
\end{proof}

\section{Zariski density of the geometric exceptional set and the counterexamples}
\subsection{Zariski density of the geometric exceptional set}\label{sec5.1}
\

%In this section we prove Corollary $\ref{cor1.5}$. 
In this section we prove Theorem $\ref{thm2}$.
Again we work over a field which contains $\bQ(\sqrt{-3})$ and follow the notation in Theorem $\ref{the1}$.
%Since $S$ is a del Pezzo surface of degree $1$ and Picard rank $1$ (Theorem $\ref{the1}$ (1)), from the discussion in Section $\ref{sec3.2}$, we know that rational points on rational curves in $\lvert-2K_S\rvert$ are contained in the exceptional set.
\begin{proof}[Proof of Theorem $\ref{thm2}$]
	By Proposition \ref{propgesofdp1}, rational points on rational curves in $\lvert-2K_S\rvert$ are contained in the geometric exceptional set if we assume the Picard rank of $S$ is $1$.
	So it is sufficient to show that the set of these rational points is dense in $S$.
	
	%Since the Mordell-Weil rank of $E$ in Theorem $\ref{the1}$ is $1$, there are infinite many rational points on $E$ and they form a dense subset of $E$.
	The elliptic curve $E$ has a dense set of rational points since it has positive rank.
	%Since the map
	%$$\{\text{rational points on } E \}\mapsto\{ \text{rational curves in $\lvert-2K_S\rvert$ defined over $k$} \}$$
	%is injective, we have $\dim(\overline{\mu(\cC)})\geq 2$.
	%But $S$ is irreducible of dimension $2$, so we must have $\overline{\mu(\cC)}=S$.
	%Namely, $\mu$ is a dominant morphism.
	%Thus we only need to show rational points on $\cC$ are Zariski dense.
	Since $\mu$ is dominant, we only need to show a general member in $\cC$ has a dense set of rational points.
	%Since the inverse image of a dense set is dense, it is sufficient to show the rational curve $\pi^{-1}(p)$ has Zariski dense rational points.
	It is well-known that a rational curve has infinitely many rational points if and only if it has a smooth rational point \cite[Theorem A.4.3.1]{HS13}.
	Thus by Theorem $\ref{the1}$ (3), the proposition follows.
\end{proof}

\subsection{Existence of the surface}
In this section we prove Proposition \ref{prop1.6}.
\begin{proof}[Proof of Proposition \ref{prop1.6}]
	When $m=7$ and $k=\bQ(\sqrt{-3})$, we have confirmed that the Picard rank of $S$ is $1$ in Section \ref{secgalois}.
	
	%	To construct an infinite family of such surfaces $S$, we need an infinite family of integers $m$ satisfying the assumptions in Theorem \ref{thm2}.
	%	This is summarized in the following Lemma.
	The last part of the proposition is summarized in the following lemma.
\end{proof}
\begin{lemm}
	In the notation of Theorem \ref{the1}, we have
	\begin{enumerate}[(1)]
		\item If the prime-to-$2$-part of the integer $m$ is not a cube, then the Picard rank of $S$ is $1$.
		\item If we assume the Tate-Shafarevich groups of elliptic curves over $\bQ$ with $j$-invariant $0$ are finite, then the Mordell-Weil rank of $E$ is odd when 
		$$m\in\left\{ p: \textrm{\rm prime number}\mid p\equiv 4,7 \textrm{ \rm or }8 \,\mod\, 9 \right\}.$$
	\end{enumerate}
\end{lemm}
\begin{proof}
	Let $S(m)$ denote the surface $S$ defined by $m$ in Theorem \ref{the1}.
	Set $m_0=7$, then we have an isomorphism 
	$$f: \overline{S(m_0)}\rightarrow \overline{S(m)},\quad (x:y:z:w)\mapsto (\sqrt[3]{m/m_0}\,x:\sqrt[3]{m/m_0}\,y:z:w)$$
	over an algebraically closure $\overline{k}$.
	The Galois action $\rho:\sqrt[3]{m}\mapsto \zeta\sqrt[3]{m}$ on the Picard groups commute with $f$.
	So as long as the prime-to-$2$-part of the integer $m$ is not a cube, $\rho$ is still alive and in particular the Picard rank of $S(m)$ equals to $1$ as computed in Section $\ref{secgalois}$.
	
	Assuming finiteness of Tate-Shafarevich groups, \cite{nekovar2001parity} and \cite{dokchitser2010birch} show that the root number $W(E)$ of an elliptic curve $E$ over $\bQ$ is $(-1)^{\rk (E)}$ (the parity conjecture).
	
	The Weierstrass equation of $E$ defined in Proposition \ref{proppart2} is
	$$Y^2=X^3-432m^2,$$
	whose root number has been computed in \cite[Section 3.1.1]{VAphd}.
	By an elementary calculation, we find that when $m$ is a prime number $p$ such that $p\equiv 4,7 \textrm{ \rm or }8 \,\mod\, 9$, we always have $W(E)=-1$, and thus $\rk(E)$ is odd.
\end{proof}

\subsection{Counterexamples to the original version of Manin's Conjecture}
\

In this section we prove Corollary $\ref{cor1.6}$.
%Again we work over the field $k=\bQ(\zeta)$ as in Theorem $\ref{the1}$.
Our result is based on the following theorem.
\begin{theo}\label{the5.1}
	{\rm (\cite[Corollary 1.3]{BL16})} Let $X$ be a Fano variety over a number field $F$ and set $L=-K_X$.
	Suppose that the union of the saturated subvarieties $Y\subset X$ over $K$ is Zariski dense in $X$.
	Then for any non-empty open subset $U\subset X$ and any integer $A\geq 0$ there exists an adelic metric on $L$ such that
	\begin{equation*}
		\liminf_{T\rightarrow\infty}\frac{N(U,\cL,T)}{T}>(1+A)c(F,Z,\cL),
	\end{equation*}
	where $\cL$ denotes $L$ equipped with the adelic metric.
\end{theo}
The notion of saturated subvarieties is defined in \cite[Definition 1.1]{BL16}.

\begin{proof}[Proof of Corollary \ref{cor1.6}]
	In our case, an irreducible curve $f:C\rightarrow S$ is saturated if and only if there exists a metrization on $L$ such that
	\begin{equation*}
		0<\liminf_{B\rightarrow\infty}\frac{N(C,f^\ast \cL,B)}{B}<\infty.
	\end{equation*}
	Since Manin's Conjecture is confirmed for $\bP^1$ for any metrized ample line bundle without removing an exceptional set \cite{Schanuel1979},
	%	any rational curve $f:C\rightarrow S$ with $$(a(S,L),b(k,S,L))\geq(a(C,f^\ast L),b(k,C,f^\ast L))$$ in lexicographical order is a saturated subvariety of $S$.
	any rational curve in $\lvert-K_S\rvert$ and $\lvert-2K_S\rvert$ is a saturated subvariety of $S$. Union of such curves is nothing but the geometric exceptional set of $S$ by Proposition \ref{propgesofdp1}, which is dense by Theorem \ref{thm2}. Thus Theorem \ref{the5.1} applies.
	%	Thus by the discussion in Section $\ref{sec3.2}$, rational curves in $\lvert-2K_S\rvert$ are saturated.
	%	Since we have constructed a family of such curves in Theorem $\ref{the1}$ (2) and proved that the union of them is dense in $S$ in Section $\ref{sec5.1}$, the conditions of Theorem $\ref{the5.1}$ are satisfied and Corollary $\ref{cor1.6}$ follows.
\end{proof}

\section{Further computations}
%To understand the (conjectural) exceptional set of $S$ defined in Theorem $\ref{the1}$ better,
Towards a further understanding of the geometric exceptional set for $S$, we use a systematic method to compute rational curves in $\lvert-2K_S\rvert$. Though it seems difficult to determine all of them due to the computational complexity, we obtained four elliptic families of them, including the one described in Section 4.2.
All computer calculations were carried out using Magma \cite{magma}.

As shown in Section \ref{sec4.1}, rational curves in  $\lvert-2K_S\rvert$ correspond to bitangent planes of $B$, where $B$ is a curve of degree $6$ on a quadric cone $Q\subset\bP^3$. The projective dual $B^{\vee}\subset(\bP^3)^\vee$ parametrizes all planes tangent to $B$.  It has been shown in \cite[Propsition 22]{Lubbes14} that for such a curve $B$, the singular locus $\Sing(B^{\vee})$ of $B^{\vee}$ consists of the following irreducible components (after relabeling):
\begin{enumerate}[(1)]
	\item $S_1$ is a cuspidal curve dual to osculating planes of $B$.
	\item $S_2$ is a conic dual to tritangent planes of $B$ which are tangent to $Q$.
	\item $S_i$ for $3\leq i\leq N$ are irreducible components of other bitangent planes of $B$, where $N\leq 14$. They parametrize rational curves in $\lvert-2K_S\rvert$.
\end{enumerate}
One can calculate $B^{\vee}$ explicitly in the following way. Without loss of generality, we may let the defining polynomial of $S$ be
$$w^2=z^3+x^6+y^6.$$
Then $B$ is defined by polynomials
$$U:=X^3+Z^3+W^3\quad \text{and}\quad V:=XZ-Y^2$$
in $\bP^3=\bP(X,Y,Z,W)$. A plane in $\bP^3$ is given by
$$H:=kX+lY+mZ+nW,$$
which is dual to a point $(k,l,m,n)$ in the dual space $(\bP^3)^\vee$.
The polynomials $U,V$ and $H$ define a variety $T$ in $\bP^3\times (\bP^3)^\vee$.
Let $M:=\partial(U,V,H)/\partial(x,y,z,w)$ be the Jacobian matrix of $T$ relative to the first $\bP^3$.
The minors of order $3$ of $M$ together with $U,V,H$ define a subvariety $T_0$ of $T$. Then the dual variety $B^{\vee}$ is nothing but the image of $T_0$ along the canonical projection $\bP^3\times (\bP^3)^\vee\rightarrow (\bP^3)^\vee$.

It turns out that the surface $B^{\vee}$ is defined by a polynomial of degree $18$. Its singular locus $\Sing B^{\vee}$ is a curve of degree $170$ defined by $4$ polynomials.
It seems impractical to compute irreducible components of $\Sing B^{\vee}$ directly.
Nonetheless, following the same idea of finding involutions of $B$ as in Section $\ref{sec4.2}$, we found four of them in case $(3)$:
\begin{itemize}
	\item $S _3:=V(l,\ k^6 - 2k^3m^3 - 2k^3n^3 + m^6 - 2m^3n^3 + n^6)$,
	\item $S _4:=V(k - m,\ 
	l^6 - 6l^4m^2 + 9l^2m^4 - 12l^2mn^3 + 4m^3n^3 - 4n^6)$,
	\item $S _5:=V(k + (\zeta  + 1)m,\ 
	l^6 + 6(\zeta  + 1)l^4m^2 + 9\zeta l^2m^4 - 12\zeta l^2mn^3 + 4m^3n^3 - 4n^6)$,
	\item $S_6:=V(k - \zeta m,\  
	l^6 - 6\zeta l^4m^2 -9 (\zeta  +1)l^2m^4 + 12(\zeta  + 1)l^2mn^3 + 4m^3n^3 -
	4n^6)$,
\end{itemize}
where $\zeta$ is a cube root of unity.
They are all elliptic components where $S_3$ is the one described in Section $\ref{sec4.2}$ (after a scalar multiplication of variables).
We found these irreducible components by considering the following involutions of $B$ respectively:
\begin{itemize}
	\item $\iota_3:(X:Y:Z:W)\mapsto (X:-Y:Z:W)$,
	\item $\iota_4:(X:Y:Z:W)\mapsto (Z:Y:X:W)$,
	\item $\iota_5:(X:Y:Z:W)\mapsto (\zeta Z:Y:\zeta^2 X:W)$,
	\item $\iota_6:(X:Y:Z:W)\mapsto (\zeta^2 Z:Y:\zeta X:W)$,
\end{itemize}
where $\iota_3$ is the same involution considered in Section $\ref{sec4.2}$.
In fact, these four involutions $\iota_i$ have the same property that the direction of the vector $\iota_i(P)-P$ does not vary with respect to $P\in B$.
%One can check that each involution of $B$ with this property can produce a family of bitangent planes to $B$.
A direct calculation by Magma \cite{magma} shows that the irreducible curves $S_3,S_4,S_5, S_6$ are all of geometric genus $1$. So $S_4,S_5$ and $S_6$ have the potential to produce dense subsets of the geometric exceptional set as well.

The total degree of $S_1,S_2,\cdots,S_6$ turns out to be $70$, so there are still not more than $8$ other components of $\Sing B^{\vee}$ of total degree $100$.

\bibliographystyle{alpha}
\bibliography{dp1-bibliography}% common bib file

\end{document}